\documentclass{amsart}

\usepackage{amssymb}
\usepackage[all,cmtip]{xy}
\usepackage{tikz-cd}
\usetikzlibrary{matrix}
\usepackage{hyperref}

\newcommand{\isomto}{\overset{\sim}{\to}}
\newcommand{\longisomto}{\overset{\sim}{\longrightarrow}}

\newcommand{\longto}{\longrightarrow}

\makeatletter
\newcommand*\rel@kern[1]{\kern#1\dimexpr\macc@kerna}
\newcommand*\widebar[1]{%
  \begingroup
  \def\mathaccent##1##2{%
    \rel@kern{0.8}%
    \overline{\rel@kern{-0.8}\macc@nucleus\rel@kern{0.2}}%
    \rel@kern{-0.2}%
  }%
  \macc@depth\@ne
  \let\math@bgroup\@empty \let\math@egroup\macc@set@skewchar
  \mathsurround\z@ \frozen@everymath{\mathgroup\macc@group\relax}%
  \macc@set@skewchar\relax
  \let\mathaccentV\macc@nested@a
  \macc@nested@a\relax111{#1}%
  \endgroup
}
\makeatother

\DeclareMathOperator\Frac{Frac}

\DeclareMathOperator\Hom{Hom}
\DeclareMathOperator\End{End}
\DeclareMathOperator\Spec{Spec}

\DeclareMathOperator\Art{Art}

\DeclareMathOperator\rk{rk}

\DeclareMathOperator\Pic{Pic}

\DeclareMathOperator\Cov{Cov}
\DeclareMathOperator\Br{Br}

\DeclareMathOperator\GL{GL}

\DeclareMathOperator\SO{SO}
\DeclareMathOperator\rO{O}

\DeclareMathOperator\Gal{Gal}

\DeclareMathOperator\Ext{Ext}
\DeclareMathOperator\Tor{Tor}

\DeclareMathOperator\Fil{Fil}

\DeclareMathOperator\Ab{Ab}
\DeclareMathOperator\Nm{Nm}

\DeclareMathOperator\tr{tr}
\DeclareMathOperator\colim{colim}
\DeclareMathOperator\cocolim{lim}
\DeclareMathOperator\Frob{Frob}

\def\bQ{{\mathbf{Q}}} \def\bZ{{\mathbf{Z}}} 
\def\bF{{\mathbf{F}}} \def\bG{{\mathbf{G}}} \def\bR{{\mathbf{R}}}
\def\bC{{\mathbf{C}}}  
\def\bA{{\mathbf{A}}} 

\def\cO{{\mathcal{O}}} \def\cU{{\mathcal{U}}} 
  \def\cL{{\mathcal{L}}} 
  
 \def\cD{{\mathcal D}} 
\def\cK{{\mathcal K}}  \def\cX{{\mathcal{X}}}
\def\cW{{\mathcal W}}

 \def\rR{{\rm R}} \def\rH{{\rm H}}

\def\fm{{\mathfrak{m}}} \def\fX{{\mathfrak{X}}}

\def\ab{{\rm ab}}
\def\id{{\rm id}}

\def\et{{\rm \acute{e}t}}
\def\fl{{\rm fl}}
\def\nr{{\rm nr}}

\def\rec{{\rm rec}}

\def\dR{{\rm dR}}

\def\can{{\rm can}}
\def\crys{{\rm crys}}

\def\NS{{\rm NS}}

\theoremstyle{plain}
\newtheorem{theorem}{Theorem}[section]
\newtheorem{corollary}[theorem]{Corollary}
\newtheorem*{corollary*}{Corollary}
\newtheorem{lemma}[theorem]{Lemma}
\newtheorem{bigthm}{Theorem}
\newtheorem{proposition}[theorem]{Proposition}
\theoremstyle{definition}

\newtheorem*{definition*}{Definition}

\newtheorem{remark}[theorem]{Remark}

\newtheorem*{question*}{Question}

\begin{document}

\title
 {Ordinary K3 surfaces over a finite field}

\author
 {Lenny Taelman}

\begin{abstract}
 We give a description of the category of ordinary K3 surfaces over a finite field in terms of linear algebra data over $\bZ$. This gives an analogue for K3 surfaces of Deligne's description of the category of ordinary abelian varieties over a finite field, and refines earlier work  by N.O.~Nygaard and J.-D.~Yu. 
 
Our main result is conditional on a conjecture on potential semi-stable reduction of K3 surfaces over $p$-adic fields. We give unconditional versions for K3 surfaces of large Picard rank and for K3 surfaces of small degree.
 \end{abstract}

\maketitle

\section*{Introduction}

\subsection*{Statement of the main results}
A K3 surface $X$ over a perfect field $k$ of characteristic $p$ is called \emph{ordinary} if it satisfies the following equivalent conditions:
\begin{enumerate}
\item the Hodge and Newton polygons of $\rH^2_\crys(X/W(k))$ coincide,
\item the Frobenius endomorphism of $\rH^2(X,\cO_X)$ is a bijection,
\item the formal Brauer group  of $X$ (see \cite{ArtinMazur77}) has height $1$.
\end{enumerate}
If $k$ is finite, then these are also equivalent with $|X(k)| \not\equiv 1 \bmod p$. Building on \cite{ArtinMazur77} and \cite{DeligneIllusie81}, Nygaard \cite{Nygaard83} has shown that such an ordinary $X$ has a canonical lift $X_\can$ over the ring of Witt vectors $W(k)$. 


Choose an embedding $\iota\colon W(\bF_q) \to \bC$. Then with every ordinary K3 surface over $\bF_q$ we can associate a complex K3 surface $X^\iota_\can := X_\can \otimes_{W(\bF_q)} \bC$ and an integral
lattice
\[
	M :=  \rH^2(X_{\can}^\iota,\,\bZ).
\]
Using the Kuga-Satake construction, Nygaard \cite{Nygaard83} and Yu \cite{Yu12} have shown that there exists a (necessarily unique) endomorphism $F$ of $M\otimes_\bZ \bZ[\tfrac{1}{p}]$ such that for every $\ell\neq p$ the canonical isomorphism
\[
	M\otimes \bZ_\ell \longisomto \rH^2_\et(X_{\bar \bF_q},\bZ_\ell)
\]
matches $F$ with the geometric Frobenius $\Frob$ on \'etale cohomology (see also \S~\ref{sec:definition-of-F}).
We have:
\begin{enumerate}
\item[(M1)] the pairing $\langle -,- \rangle$ on $M$ is unimodular, even, and of signature $(3,19)$;
\item[(M2)] $\langle Fx, Fy \rangle = q^2 \langle x,  y\rangle$ for every $x,y\in M$.
\end{enumerate}
From Deligne's proof of the Weil conjectures for K3 surfaces \cite{Deligne72} we also know 
\begin{enumerate}
\item[(M3)] the endomorphism $F$ of $M\otimes \bC$ is semi-simple and all its eigenvalues have absolute value $q$.
\end{enumerate}
Our first result is an integral $p$-adic property of the pair $(M,F)$.

\begin{bigthm}\label{bigthm:properties}
The endomorphism $F$ preserves the $\bZ$-module $M$ and satisfies
\begin{enumerate}
\item[(M4)] the $\bZ_p[F]$-module $M\otimes \bZ_p$ decomposes as 
$M^{0} \oplus M^{1} \oplus M^{2}$
with 
\begin{enumerate}
\item $FM^s=q^s M^s$, for all $s$,
\item $M^0$, $M^1$ and $M^2$ are free $\bZ_p$-modules of rank $1$, $20$, $1$, respectively.
\end{enumerate}
\end{enumerate}
\end{bigthm}

For a $\bZ$-lattice $M$ equipped with an endomorphism $F$ satisfying (M1)--(M4) we denote by $\NS(M,F)$ the group
\[
	\NS(M,F) := \{ x \in M \mid F^d x = q^d x \text{ for some $d\geq 1$} \}.
\]
Using the fact that all line bundles on an ordinary K3 surface extend uniquely to its canonical lift, one shows  that there is a natural bijection $\Pic X_{\bar \bF_q} \to \NS(M,F)$, and that the ample line bundles
on $X_{\bar \bF_q}$ span a real cone $\cK \subset M\otimes \bR$ satisfying 
\begin{enumerate}
\item[(M5)] $\cK$ is a connected component of
\[
\Big\{\, x \in \NS(M,F)\otimes\bR \,\Bigm|\, \langle x, x\rangle >0,\, \langle x,\delta \rangle \neq 0 \text{ for all $\delta \in \NS(M,F)$ with $\delta^2=-2$}\,\Big\}
\]
satisfying $F\cK  = \cK$.
\end{enumerate}
See \S~\ref{sec:line-bundles} for more details.


\begin{definition*}[see \cite{LiedtkeMatsumoto18}] Let $\cO_K$ be a complete discrete valuation ring with fraction field $K$. We say that a K3 surface $X$ over $K$ satisfies ($\star$) if
there exists a finite extension $K\subset L$ and algebraic space $\fX/\cO_L$ satisfying
\begin{enumerate}
\item $\fX_L \cong X_L$,
\item $\fX$ is regular,
\item the special fiber is a strict normal crossing divisor in $\fX$,
\item the relative dualizing sheaf of $\fX/\cO_L$ is trivial.
\end{enumerate}
\end{definition*}
This is a strong form of `potential semi-stable reduction'. Over a complete dvr $\cO_K$ of residue characteristic $0$, it is known that all K3 surfaces satisfy ($\star$). It is expected to hold in general, but currently only known under extra assumptions on~$X$. 

Our main result is the following description of the category of ordinary K3 surfaces over $\bF_q$. It is an analogue for K3 surfaces of Deligne's theorem \cite{Deligne69} on ordinary abelian varieties over a finite field.

\begin{bigthm}\label{bigthm:equivalence}
Fix an embedding $\iota\colon W(\bF_q)\to \bC$. Then the resulting functor $X\mapsto (M,F,\cK)$ is a fully faithful functor between
the groupoids of
\begin{enumerate}
\item ordinary K3 surfaces $X$ over $\bF_q$, and
\item triples $(M,F,\cK)$ consisting of
\begin{enumerate}
\item an integral lattice $M$,
\item an endomorphism $F$ of the $\bZ$-module $M$, and
\item a convex subset $\cK \subset M\otimes \bR$, 
\end{enumerate}
satisfying (M1)--(M5).
\end{enumerate}
If every K3 surface over $\Frac W(\bF_q)$ satisfies ($\star$) then the functor is 
essentially surjective.
\end{bigthm}

Fully faithfulness is essentially due to Nygaard \cite{Nygaard83b} and Yu \cite{Yu12}. Our contribution is a description of the image of this functor.

Restricting to families for which ($\star$) is known to hold, we also obtain unconditional equivalences of categories between ordinary K3 surfaces $X/\bF_q$ satisfying one of the following
additional conditions
\begin{enumerate}
\item there is an ample $\cL \in \Pic X_{\bar \bF_q}$ with $\cL^2 < p-4$,
\item $\Pic X_{\bar \bF_q}$ contains a hyperbolic plane and $p\geq 5$,
\item $X$ has geometric Picard rank $\geq 12$ and $p\geq 5$,
\end{enumerate}
and triples $(M,F,\cK)$ satisfying the analogous constraints.  See Theorem \ref{thm:unconditional} for the precise statement.

\subsection*{About the proofs}

A crucial ingredient in the proof of Theorems \ref{bigthm:properties} and \ref{bigthm:equivalence} is the following criterion which distinguishes the canonical lift amongst all lifts using $p$-adic \'etale cohomology.

\begin{bigthm}\label{bigthm:etale-characterisation}
Let $\cO_K$ be a complete discrete valuation ring with perfect residue field $k$ of characteristic $p$ and fraction field $K$ of characteristic $0$. Let $\fX$ be a projective K3 surface over $\cO_K$ and assume that $\fX_k$ is ordinary. Then the following are equivalent:
\begin{enumerate}
\item $\fX$ is the base change from $W(k)$  to $\cO_K$ of  the canonical lift of $\fX_k$
\item $\rH^2_\et(\fX_{\bar K},\bZ_p) \cong H^0 \oplus H^1(-1) \oplus H^2(-2)$ with $H^i$ unramified $\bZ_p[\Gal_K]$-modules, free of rank $1$, $20$, $1$ over $\bZ_p$, respectively.
\end{enumerate}
\end{bigthm}

Here the $(-1)$ and $(-2)$ in (ii) denote Tate twists. This theorem is an integral refinement of a theorem of Yu \cite{Yu12} which characterizes \emph{quasi-canonical} lifts by the splitting of \'etale cohomology with $\bQ_p$-coefficients.

The canonical lift of $X$ is defined in terms of its enlarged formal Brauer group $\Psi$ (a $p$-divisible group), and to prove Theorem \ref{bigthm:etale-characterisation} we need to compare the $p$-adic \'etale cohomology of the generic fiber of a K3 surface over $\cO_K$ to the Tate module of its enlarged formal Brauer group. With $\bQ_p$-coefficients, such comparison has been shown by Artin and Mazur \cite{ArtinMazur77}. We give a different argument leading to an integral version, see Theorem~\ref{thm:formal-brauer-etale-cohomology}.  Once Theorem \ref{bigthm:etale-characterisation} is established, Theorem \ref{bigthm:properties} is an almost formal consequence.

Finally, we briefly sketch the argument for the proof of Theorem \ref{bigthm:equivalence}. Fully faithfulness was shown by Nygaard \cite{Nygaard83b} and Yu \cite{Yu12} (see \S~\ref{sec:fully-faithful} for more details). The proof of essential surjectivity starts with the observation that the decomposition in (M4) induces (via the embedding $\iota\colon \bZ_p \to \bC$) a Hodge structure on $M$, for which $\NS(M,F)$ consists precisely of the Hodge classes.  The Torelli theorem for K3 surfaces then shows that there is a canonical K3 surface $X/\bC$ with $\rH^2(X,\bZ)=M$ and whose ample cone $\cK \subset \rH^2(X,\bR)$ coincides with  $\cK \subset M\otimes \bR$. This K3 surface has complex multiplication, and hence can be defined over a number field. Using the strong version of the main theorem of CM for K3 surfaces of \cite{Taelman17} we show that we can find a model of $X$ over $K := \Frac W(\bF_q) \subset \bC$  such that
\begin{enumerate}
\item the $\Gal_{K}$-module $\rH^2_\et(X_{\bar K},\bZ_p) = M \otimes \bZ_p$ decomposes as in Theorem \ref{bigthm:etale-characterisation},
\item for $\ell\neq p$, the $\Gal_{K}$-module $\rH^2_\et(X_{\bar K},\bZ_\ell)=M\otimes \bZ_\ell$ is unramified, and Frobenius acts as $F$.
\end{enumerate}
Assuming ($\star$), it follows from N\'eron--Ogg--Shafarevich criterion of Liedtke and Matsumoto \cite{LiedtkeMatsumoto18} that $X$ has good reduction over an unramified extension $L$ of $K$. Using  Theorem \ref{bigthm:etale-characterisation} we show that $X_L$ is the canonical lift of its reduction, and deduce from this that $X$ has already a smooth projective model $\fX$ over $\cO_K$. By construction, its reduction $\fX_k$ maps under our functor to the given triple $(M,F,\cK)$.

\subsection*{A question}

We end this introduction with an essentially lattice-theoretical question to which we do not know the answer:

\begin{question*}Does there exist a triple $(M,F,\cK)$ satisfying (M1)--(M5) and the inequality $1+ \tr F + q^2  <0$?\end{question*} 

By (M3) such triple can only exist for small $q$. A positive answer to this question would imply that there exist K3 surfaces over $p$-adic fields that do not satisfy ($\star$). Indeed, if $(M,F,\cK)$ came from a K3
surface $X/\bF_q$ as in Theorem~\ref{bigthm:equivalence} we would have
\[
	|X(\bF_q)| = \tr(\Frob,\, \rH^\bullet(X_{\bar \bF_q},\,\bQ_\ell)) = 1 + \tr F + q^2 < 0,
\]
which is absurd.

\section{$p$-divisible groups associated to K3 surfaces}

Let $\Lambda$ be a complete noetherian local ring with perfect residue field $k$ of characteristic $p>0$ 
and let $\fX$ be a K3 surface over $\Spec \Lambda$. We recall (and complement) some of the main results of Artin and Mazur \cite{ArtinMazur77} on the formal Brauer group and enlarged formal Brauer group of $\fX$.

\subsection{The formal Brauer group}
 Let $\Art_\Lambda$ be the category of Artinian local $\Lambda$-algebras $(A,\fm)$ with perfect residue field $A/\fm$. For an $(A,\fm) \in \Art_\Lambda$ we denote by $\cU_A$ the sheaf on $\fX_\et$ defined by the exact sequence 
\[
	1 \to \cU_A \to \cO_{\fX_A}^\times \to \cO_{\fX_{A/\fm}}^\times \to 1.
\]
The \emph{formal Brauer group} of $\fX$ is the functor
\[
	\hat\Br(\fX)\colon \Art_\Lambda \to \Ab,\,
	A \mapsto \rH^2(\fX_\et, \cU_A ).
\]
By \cite{ArtinMazur77} it is representable by a one-dimensional formal group, and if $\fX$ is not supersingular then $\hat\Br(\fX)$ is a $p$-divisible group. 

\begin{lemma}\label{lemma:etale-to-zariski}
$\rH^\bullet(\fX_\et,\cU_A) = \rH^\bullet(\fX,\cU_A)$ and $\rH^1(\fX_\et,\cU_A)=0$.
\end{lemma}

Here $\rH^\bullet(\fX,-)$ denotes Zariski cohomology.

\begin{proof}[Proof of Lemma \ref{lemma:etale-to-zariski}]
The sheaf $\cU_A$ has a filtration whose graded pieces are 
\[
	\fm^n/\fm^{n+1} \otimes_{A/\fm}\cO_{\fX_{A/\fm}}
\]
Since these are coherent, we have
$\rH^\bullet(\fX_\et,\cU_A) = \rH^\bullet(\fX,\cU_A)$. Moreover, since $\rH^1(\fX_{A/\fm},\cO_{\fX_{A/\fm}})$ vanishes,
we conclude that  $\rH^1(\fX_\et,\cU_A)=0$.
\end{proof}

\begin{lemma}\label{lemma:torsion-of-formal-brauer}
For every $(A,\fm)\in \Art_\Lambda$ there is a natural exact sequence
\[
	0 \longto \hat\Br(\fX)[p^r](A) \longto \rH^2_\fl(\fX_A, \mu_{p^r} ) \longto \rH^2_\fl(\fX_{A/\fm},\mu_{p^r})
\]
of abelian groups.
\end{lemma}

Here $\rH^\bullet_\fl$ denotes cohomology in the fppf topology.

\begin{proof}[Proof of Lemma \ref{lemma:torsion-of-formal-brauer}]
Consider the complex $\cU_A \overset{p^r}{\to} \cU_A$ on $\fX_\et$ in degrees $0$ and $1$. We have a short exact sequence
\[
	1 \longto \rH^1(\fX,\cU_A) \otimes \bZ/p^r\bZ \longto \rH^2(\fX,\cU_A \overset{p^r}{\to} \cU_A ) \longto
	\rH^2(\fX,\cU_A)[p^r] \longto 1
\]
and thanks to Lemma \ref{lemma:etale-to-zariski} we obtain canonical isomorphisms
\[
	  \hat\Br(\fX)[p^r](A)  = \rH^2(\fX,\,\cU_A)[p^r] = \rH^2(\fX,\cU_A \overset{p^r}{\to} \cU_A ).
\]
Now consider the short exact sequence
\[
\begin{tikzcd}
1 \arrow{r} & \cU_A \arrow{d}{p^r} \arrow{r}
	& \cO_{\fX_A}^\times \arrow{d}{p^r} \arrow{r}
	&  \cO_{\fX_{A/\fm}}^\times \arrow{d}{p^r} \arrow{r} & 1 \\
1 \arrow{r} & \cU_A  \arrow{r}
	& \cO_{\fX_A}^\times \arrow{r}
	&  \cO_{\fX_{A/\fm}}^\times \arrow{r} & 1
\end{tikzcd}
\]
of length-two complexes concentrated in degrees $0$ and $1$.
Using the Kummer sequence and the fact that $\bG_m$ is smooth, we see that
\[
	\rH^n_\et(\fX, \cO_{\fX_A}^\times \overset{p^r}\to  \cO_{\fX_A}^\times )
	= \rH^n_\fl(\fX_{A}, \bG_m \overset{p^r}\to \bG_m ) = \rH^n_\fl(\fX_A,\mu_{p^r})
\]
and similarly for $\fX_{A/\fm}$. It follows that the above short exact sequence of complexes induces a long exact sequence of (hyper-)cohomology groups
\[
	\longto \rH^1_\fl(\fX_{A/\fm},\mu_{p^r}) \longto \rH^2(\fX,\cU_A \overset{p^r}{\to} \cU_A )
	\longto \rH^2_\fl(\fX_A, \mu_{p^r} ) \longto \rH^2_\fl(\fX_{A/\fm}, \mu_{p^r}) \longto
\]
Since $A/\fm$ is perfect and $\Pic \fX_{A/\fm}$ is torsion-free, we have $\rH^1_\fl(\fX_{A/\fm},\mu_{p^r})=0$ and we 
conclude
\[
	 \hat\Br(\fX)[p^r](A) =  \rH^2(\fX,\cU_A \overset{p^r}{\to} \cU_A ) = 
	 \ker \big[ \rH^2_\fl(\fX_A, \mu_{p^r} ) \to \rH^2_\fl(\fX_{A/\fm}, \mu_{p^r})  \big],
\]
which is what we had to show.
\end{proof}

\subsection{The enlarged formal Brauer group}
We now assume that $\fX_k$ is ordinary.
Denote by $\mu_{p^\infty}$ the sheaf $\colim_r \mu_{p^r}$ on the fppf site. The \emph{enlarged formal Brauer group} of $\fX$ is the functor
\[
	\Psi(\fX)\colon \Art_\Lambda \to \Ab,\, A \mapsto \rH^2_\fl(\fX_A,\mu_{p^\infty}).
\]
A priori this differs from the definition of Artin--Mazur \cite[\S~IV.1]{ArtinMazur77} in two ways. First, Artin and Mazur restrict to $A$ with algebraically closed residue fields, and then use Galois descent to extend their definition to non-closed  perfect residue fields, and second, they restrict to those classes in $\rH^2_\fl(\fX_A,\mu_{p^\infty})$ that map to the $p$-divisible part of $\rH^2_\fl(\fX_{\bar k},\mu_{p^\infty})$. Lemmas \ref{lemma:ordinary-p-divisible} and \ref{lemma:galois-invariants} below show that  the above definition is equivalent to that of Artin--Mazur (under our standing condition that $\fX_k$ is an ordinary K3 surface). See also \cite[Cor.~1.5]{Nygaard83}.

The following lemma is well-known, and implicitly used in \cite{ArtinMazur77} and \cite{Nygaard83}. We include it for the sake of completeness.

\begin{lemma}
For any quasi-compact quasi-separated scheme $X$ there is a natural isomorphism
\[
	\rH^\bullet_\fl(X,\mu_{p^\infty}) \longisomto \colim_r \rH^\bullet_\fl(X,\mu_{p^r}).
\]
\end{lemma}

\begin{proof}
This follows from \cite[0739]{stacks}, taking for $\mathcal B$ the class of quasi-compact and quasi-separated schemes, and for $\Cov$ the fpqc covers consisting of finitely many affine schemes. 
\end{proof}

\begin{lemma}[{\cite[Cor.~1.4]{Nygaard83}}] \label{lemma:ordinary-p-divisible}
The group $\rH^2_\fl(\fX_{\bar k}, \mu_{p^\infty})$ is $p$-divisible. \qed
\end{lemma}

\begin{lemma}\label{lemma:galois-invariants}
$\rH^2_\fl(\fX_{\bar k}, \mu_{p^\infty})^{\Gal(\bar k/k)} = \rH^2_\fl(\fX_{k}, \mu_{p^\infty})$.
\end{lemma}

\begin{proof}Since the $p$-th power map $\bar k^\times \to \bar k^\times$ is a bijection, and since $\Pic \fX_{\bar k}$ is torsion-free, we have
\[
	\rH^i_\fl(\fX_{\bar k},\mu_{p^\infty})= \colim_r \rH^i_\fl(\fX_{\bar k},\mu_{p^r}) = 0
\]
for $i\in \{0,1\}$. It now follows from the Hochschild-Serre spectral sequence
\[
	E_2^{s,t} = \rH^s\big(\Gal(\bar k/k), \,\rH^t_\fl(\fX_{\bar k},\mu_{p^\infty})\big) \Rightarrow \rH^{s+t}_\fl(\fX_k,\mu_{p^\infty})
\]
that $\rH^2_\fl(\fX_{\bar k}, \mu_{p^\infty})^{\Gal(\bar k/k)} = \rH^2_\fl(\fX_{k}, \mu_{p^\infty})$.
\end{proof}

The following theorem summarizes the properties of the enlarged formal Brauer group that we will use.

\begin{theorem}
Let $\fX$ be a formal K3 surface over $\Lambda$ with $\fX_k$ ordinary. Then the enlarged formal Brauer group $\Psi(\fX)$ is representable by a $p$-divisible group over $\Lambda$. Its \'etale-local exact sequence 
\[
	0 \to \Psi^\circ(\fX) \to \Psi(\fX) \to \Psi^\et(\fX) \to 0
\]
satisfies 
\begin{enumerate}
\item $\Psi^\circ(\fX)$ is a connected $p$-divisible group of height $1$, with 
\[
	\Psi^\circ(\fX)[p^r](A) = 
	\ker \big(\rH^2_\fl(\fX_A,\mu_{p^r}) \to \rH^2_\fl(\fX_{A/\fm},\mu_{p^r})\big) 
\]
for all $(A,\fm)\in \Art_\Lambda$. It is canonically isomorphic to $\hat\Br(\fX)$;
\item $\Psi(\fX)[p^r](A) = \rH^2_\fl(\fX_A,\, \mu_{p^r})$ for all $(A,\fm)\in \Art_\Lambda$;
\item $\Psi^\et(\fX)[p^r](A) = \rH^2_\fl(\fX_{A/\fm},\,\mu_{p^r})$ for all $(A,\fm)\in \Art_\Lambda$.
\end{enumerate}
\end{theorem}

\begin{proof}The representability and (i) are shown in \cite[Prop.~IV.1.8]{ArtinMazur77}. 

To prove (ii), note that since $\rH^1(\fX_A,\,\bG_m) = \Pic \fX_A$ is torsion-free, we have a natural isomorphism
\[
	\rH^1_\fl(\fX_A,\mu_{p^r}) = \rH^0(\fX_A,\,\bG_m) \otimes \bZ/p^r\bZ.
\]
Taking the colimit over $r$ we obtain a natural isomorphism
\[
	\rH^1_\fl(\fX_A,\mu_{p^\infty}) = \colim_r \rH^1_\fl(\fX_A,\mu_{p^r})
	= \rH^0(\fX_A,\,\bG_m) \otimes (\bQ_p /\bZ_p),
\]
and in particular we see that $\rH^1_\fl(\fX_A,\, \mu_{p^{\infty}}) $ is $p$-divisible. 
Now the long exact sequence associated to
\[
	1 \longto \mu_{p^r} \longto \mu_{p^{\infty}} \overset{p^r}\longto \mu_{p^{\infty}} \longto 1
\]
induces an isomorphism $\rH^2_\fl(\fX_A,\, \mu_{p^r}) \cong \rH^2_\fl(\fX_A,\, \mu_{p^\infty})[p^r]$, which proves (ii). 

A similar argument shows (iii) for $A$ with $A/\fm$ algebraically closed, after which Lemma~\ref{lemma:galois-invariants} implies the general case.
\end{proof}

\subsection{Canonical lifts}\label{sec:canonical-lifts}

Let $X/k$ be an ordinary K3 surface. Since formal groups of height $1$ are rigid, the $p$-divisible group $\Psi^\circ(X)$ over $k$ extends uniquely to a $p$-divisible group $\Psi^\circ(X)_\can$ over $\Lambda$. Also the \'etale $p$-divisible group $\Psi^\et(X)$ over $k$ extends uniquely to a $p$-divisible group $\Psi^\et(X)_\can$ over $\Lambda$. 

To every lift $\fX/\Lambda$ of $X/k$ we then have an associated short exact 
sequence of $p$-divisible groups
\begin{equation}\label{eq:lift-etale-local}
	0 \longto \Psi^\circ(X)_\can \longto \Psi(\fX) \longto \Psi^\et(X)_\can \longto 0
\end{equation}
over $\Lambda$. In analogy with Serre-Tate theory, we have the following theorem.

\begin{theorem}[Nygaard {\cite[Thm.~1.6]{Nygaard83}}]
The map 
\[
	\big\{ \,\text{formal lifts $\fX/\Lambda$ of $X/k$} \,\big\}
	\to \Ext^1_{\Lambda}\big(\Psi^\et(X)_\can,\,\Psi^\circ(X)_\can\big),\, \fX \mapsto \Psi(\fX)
\]
is a bijection. \qed
\end{theorem}

It follows that there exists a unique lift $\fX/\Lambda$ for which the sequence (\ref{eq:lift-etale-local}) splits. This $\fX$ is unique up to unique isomorphism, and is called the \emph{canonical lift} of $X$. We denote it by $X_\can$.

\begin{proposition}[{\cite[Prop.~1.8]{Nygaard83}}]\label{prop:line-bundles-lift}
$\Pic X_\can \to \Pic X$ is a bijection. \qed
\end{proposition}

\begin{corollary}[{\cite[Prop.~1.8]{Nygaard83}}]
$X_\can$ is algebraizable and projective. \qed
\end{corollary}

\section{$p$-adic \'etale cohomology}

Let $\cO_K$ be a complete discrete valuation ring whose residue field $k$ is perfect of characteristic $p$ and whose fraction field $K$ is of characteristic $0$.

\subsection{$p$-adic \'etale cohomology and the enlarged formal Brauer group}

\begin{theorem}\label{thm:formal-brauer-etale-cohomology}
 Let $\fX$ be a projective K3 surface over $\cO_K$. Assume that $\fX_k$ is ordinary. Then there is a natural
injective map of $\Gal_K$-modules
\[
	T_p \Psi(\fX)_{\bar K} \to \rH^2_\et(X_{\bar K}, \bZ_p(1))
\]
whose cokernel is a free $\bZ_p$-module of rank $1$. 
\end{theorem}

Recall that if $\fX$ is ordinary, then $T_p \Psi(\fX)_{\bar K}$ has rank $21$. Up to possible torsion in the cokernel, Theorem~\ref{thm:formal-brauer-etale-cohomology} is shown in \cite[\S~IV.2]{ArtinMazur77}. The proof of Artin and Mazur is based on Lefschetz pencils, reducing the problem on $\rH^2$ to a statement about $\rH^1$ and torsors. We give a  proof working directly with the $\rH^2$ and their relation to Brauer groups to obtain the finer `integral' statement above. This is made possible by the theorem of Gabber and de Jong \cite{deJong04} asserting that the Brauer group and the cohomological Brauer group of a quasi-projective scheme coincide.

%

\medskip\noindent
Let $\fX$ be a formal K3 surface over $\cO_K$. We denote by $\fX_n$ the truncation $\fX_{\cO_K/\fm^n}$.

\begin{lemma}\label{lemma:completion}
For all $i$ the natural map
\[
	\rH^i_\fl(\fX,\mu_{p^r}) \to \cocolim_n \rH^i_\fl(\fX_n, \mu_{p^r})
\]
is an isomorphism.
\end{lemma}

\begin{proof}As in the proof of Lemma~\ref{lemma:torsion-of-formal-brauer}, we have
\[
	\rH^i_\fl(\fX_n,\mu_{p^r}) = 
	\rH^i_\et(\fX, \cO_{\fX_n}^\times \overset{p^r}\to  \cO_{\fX_n}^\times )
\]
and similarly 
\[
	\rH^i_\fl(\fX,\mu_{p^r}) = 
	\rH^i_\et(\fX, \cO_{\fX}^\times \overset{p^r}\to  \cO_{\fX}^\times ).
\]
Let $\cU_n$ be the kernel of $\cO_{\fX_n}^\times \to \cO_{\fX_1}^\times$ and $\cU$ the kernel of $\cO_\fX^\times \to \cO_{\fX_1}^\times$. Then by the usual d\'evissage arguments the lemma reduces to showing that 
\[
	\rH^i_\et(\fX,\cU) \to \cocolim_n \rH^i_\et(\fX,\cU_n)
\]
is an isomorphism for all $i$.

Since  the maps $\cU_{n+1} \to \cU_n$ are surjective, we have $\rR\!\lim_n \cU_n=\cU$. Since $\cU$ has a filtration with graded pieces isomorphic to $\cO_{\fX_1}$ it has cohomology concentrated in degrees $0$ and $2$. These two facts imply
\[
	\rR\Gamma_\et(\fX, \rR\!\cocolim_n \cU_n) =  \rH^0_\et(\fX,\cU) \oplus \rH^2_\et(\fX,\cU)[-2]
\]
in $\cD(\Ab)$. Similarly, we have 
\[
	\rR\!\cocolim_n \rR\Gamma_\et(\fX,\cU_n) = \cocolim_n \rH^0_\et(\fX,\cU_n) \oplus \cocolim_n \rH^2_\et(\fX,\cU_n)[-2]
\]
in $\cD(\Ab)$. As  $\rR\Gamma_\et$ commutes with $\rR\!\cocolim$, the lemma follows.
\end{proof}

\begin{corollary}\label{cor:generic-fiber}
If $\fX_k$ is ordinary, then $\Psi(\fX)(K)[p^r] = \rH^2_\fl(\fX,\mu_{p^r})$. 
\end{corollary}

\begin{proof}
Indeed, we have 
\[
	\Psi(\fX)(K)[p^r] = \cocolim_n \Psi(\fX)[p^r](\cO_K/\fm^n) = \cocolim_n  \rH^2_\fl(\fX_n,\mu_{p^r}),
\]
so the corollary follows from Lemma \ref{lemma:completion}.
\end{proof}

\begin{proposition}\label{prop:torsion-injective}
If $\fX$ is a projective K3 surface over $\cO_K$ then
for all $r$ the natural map 
$\rH^2_\fl(\fX,\mu_{p^r}) \to \rH^2_\fl(\fX_K,\mu_{p^r})$ is injective.
\end{proposition}

\begin{proof}
The Kummer sequence gives a commutative diagram with exact rows
\[
\begin{tikzcd}
	0 \rar & (\Pic \fX)\otimes \bZ/p^r\bZ  \dar \rar
		& \rH^2_\fl(\fX,\mu_{p^r}) \dar \rar
		& (\Br' \fX)[p^r] \dar \rar & 0 \\
	0 \rar & (\Pic \fX_K)\otimes \bZ/p^r\bZ \rar
		& \rH^2_\fl(\fX_K,\mu_{p^r}) \rar 
		& (\Br' \fX_K)[p^r] \rar & 0
\end{tikzcd}
\]
Since $\fX$ is projective, we have $\Br =\Br'$ for $\fX$ and $\fX_K$. 

 The left arrow in the diagram is an isomorphism since the special fiber $\fX_k$ is a principal divisor in $\fX$, so that $\Pic \fX \to \Pic \fX_K$ is an isomorphism. By \cite[Cor~1.8]{GrothendieckBrauerII} the natural maps of
 $\Br \fX$ and $\Br \fX_K$ to $\Br K(\fX_K)$ are injective, so that also the right arrow in the diagram is injective. We conclude that the middle map is injective.
\end{proof}

\begin{proof}[Proof of Theorem \ref{thm:formal-brauer-etale-cohomology}]
The proof is now formal. By Corollary \ref{cor:generic-fiber} and Proposition \ref{prop:torsion-injective} we have for every $r$ and every finite extension $K\subset L$ a canonical injection
\[
	\Psi(\fX)[p^r](L) \to \rH^2_\fl(\fX_L,\mu_{p^r}) = \rH^2_\et(\fX_L, \bZ/p^r\bZ(1) ).
\]
Taking the colimit over all $L$ we obtain a $\Gal_K$-equivariant injective map
\[
	\rho_r\colon \Psi(\fX)[p^r](\bar K) \to \rH^2_\et(\fX_{\bar K}, \bZ/p^r\bZ(1) ),
\]
and taking the limit over $r$ we obtain a $\Gal_K$-equivariant injective map
\[
	\rho\colon T_p \Psi(\fX)_{\bar K} \to  \rH^2_\et(\fX_{\bar K}, \bZ_p(1)).
\]
Denote the cokernel of $\rho$ by $Q$. Tensoring $\rho$ with $\bZ/p\bZ$ yields
an exact sequence
\[
	0 \longto \Tor(Q,\bZ/p\bZ) \longto \Psi(\fX)[p](\bar K) \overset{\rho\otimes \bZ/p\bZ}\longto \rH^2_\et(\fX_{\bar K}, \bZ/p\bZ(1) )
	\longto Q \otimes \bZ/p\bZ  \longto 0.
\]
Since $\rho_1=\rho\otimes \bZ/p\bZ$ is injective, we see that $\Tor(Q,\bZ/p\bZ)$ vanishes and that $Q$ is torsion-free.
\end{proof}

\subsection{Canonical lifts and $p$-adic \'etale cohomology}

In this section we prove Theorem \ref{bigthm:etale-characterisation}, characterizing the canonical lift in terms of $p$-adic \'etale cohomology.

\begin{lemma}\label{lemma:hyperbolic-plane}
Let $U$ be a free $\bZ_p$-module of rank $2$ and $b\colon U\times U \to \bZ_p$ a non-degenerate symmetric bilinear form. Let $L\subset U$ be a totally isotropic rank $1$ submodule. If $L$ is saturated in
\[
	U^\vee := \{ x \in  \bQ_p \otimes_{\bZ_p} U\, \mid b(x,U) \subset \bZ_p \}, 
\]
then $U^\vee=U$.
\end{lemma}

\begin{proof}
Since $L$ is saturated in $U^\vee$, it is also saturated in $U\subset U^\vee$ and we may choose a basis $(e,f)$ for $U$ with $L=\langle e\rangle$. Set $d := b(e,f)$. Since $b(e,e)=0$, the determinant of $b$ is $-d^2$. Since $e/d$ lies in $U^\vee$, we must have that $d$ is a unit and therefore $U^\vee = U$.
\end{proof}

\begin{proof}[Proof of Theorem \ref{bigthm:etale-characterisation}]
Assume that (ii) holds. Then we have
\[
	\rH^2_\et(\fX_{\bar K},\bZ_p(1)) = H^0(1) \oplus H^1 \oplus H^2(-1)
\]
with the $H^i$ unramified. Since the Tate module of a $p$-divisible group is Hodge-Tate of weights $0$ and $-1$, we have that $\Hom(T_p\Psi(\fX)_{\bar K}, H^2(-1)) =0$, and by Theorem \ref{thm:formal-brauer-etale-cohomology} we see that $T_p\Psi(\fX)_{\bar K} = H^0(1)\oplus H^1$. By Tate's theorem \cite[Thm.~4]{Tate67} this implies that $\Psi(\fX)=\Psi^0(\fX) \oplus \Psi^\et(\fX)$ with $T_p \Psi^0(\fX)_{\bar K}= H^0(1)$ and $T_p\Psi^\et(\fX)_{\bar K} = H^1$. It follows that $\fX$ is the base change of the canonical lift of $\fX_k$ to $\cO_K$.

Conversely, assume that $\fX$ is the base change of the canonical lift of $\fX_k$ to $\cO_K$. Let $H^1$ be the image of the direct summand $T_p\Psi^\et(\fX)_{\bar K}$ under the embedding $T_p\Psi(\fX)_{\bar K} \to \rH^2_\et(\fX_{\bar K}, \bZ_p(1))$ of Theorem \ref{thm:formal-brauer-etale-cohomology}. It is a primitive sub-module, and considering Hodge-Tate weights we see that the restriction of the bilinear form on $\rH^2$ to $H^1$ is non-degenerate. Let $U \subset \rH^2_\et(\fX_{\bar K}, \bZ_p(1))$ be its orthogonal complement. Then $U$ is a rank $2$ lattice over $\bZ_p$. The inclusions of $H^1$ and $U$ as mutual orthogonal complements inside the self-dual lattice $\rH^2_\et(\fX_{\bar K}, \bZ_p(1))$ induce an isomorphism
\[
	\alpha \colon U^\vee/U \isomto (H^1)^\vee/H^1
\]
and an identification
\[
	\rH^2_\et(\fX_{\bar K}, \bZ_p(1))   := \big\{ (x,y) \in U^\vee \oplus (H^1)^\vee \mid \alpha(x)=y \big\}.
\]

Consider the unramified $\Gal_K$-module $H^0 := T_p\Psi^\circ(\fX)_{\bar K} (-1)$. We have that $H^0(1)$ is a totally isotropic line in $U$.  We claim that it is  saturated in $U^\vee$. Indeed, if $x\in U^\vee$ satisfies $px \in H^0(1)$ then $(x,\alpha(x))$ defines a $p$-torsion element in the cokernel of $T_p \Psi(\fX)_{\bar K} \to \rH^2_\et(\fX_{\bar K}, \bZ_p(1))$, which must be trivial by Theorem \ref{thm:formal-brauer-etale-cohomology}. By Lemma \ref{lemma:hyperbolic-plane} we conclude that $U=U^\vee$ and that $\rH^2_\et(\fX_{\bar K}, \bZ_p(1)) =U \oplus H^1$. 

Now $U$ is a unimodular $\bZ_p$-lattice of rank $2$ containing an isotropic line. Moreover, since the intersection pairing on $\rH^2_\et(\fX_{\bar K}, \bZ_p(1))$ is even, so is the lattice $U$. It follows that there is a unique isotropic line  $H^2(-1) \subset U$ with $U=H^0(1) \oplus H^2(-1)$ and with $H^0$ and $H^2$ dual unramified representations. We find $\rH^2_\et(\fX_{\bar K}, \bZ_p(1)) =  H^0(1) \oplus H^1 \oplus H^2(-1)$, as claimed.
\end{proof}

\begin{remark}
Using the results on integral  $p$-adic Hodge theory by Bhatt, Morrow, and Scholze \cite{BMS} one can show that the splitting of $\rH^2_\et(\fX_{\bar K},\,\bZ_p)$ as in Theorem \ref{bigthm:etale-characterisation} implies an analogous splitting of the filtered crystal $\rH^2_\crys(\fX/W)$. If $p>2$, then the splitting of $\rH^2_\crys(\fX/W)$ implies that $\fX$ is the canonical lift of $\fX_k$ (see \cite{DeligneIllusie81} and \cite[Lem.~1.11, Thm.~1.12]{Nygaard83}). For $p=2$, however, the splitting of $\rH^2_\crys(\fX/W)$ is a weaker condition than the splitting of $\rH^2_\et(\fX_{\bar K},\, \bZ_p)$, see also \cite[2.1.16.b]{DeligneIllusie81}.
\end{remark}

\section{The functor $X \mapsto (M,F,\cK)$}

Let $\bF_q$ be a finite field with $q=p^a$ elements. Let $W$ be the ring of Witt vectors of $\bF_q$, and $K$ its fraction field. Fix an embedding $\iota\colon \bar K \to \bC$. By \S~\ref{sec:canonical-lifts}, every ordinary K3 surface $X$ over $\bF_q$ has a canonical lift $X_\can$ over $W$. We will denote by $X_\can^\iota$ the complex K3 surface obtained by base changing $X_\can$ along $\iota\colon W\to \bC$.

\subsection{Construction of a pair $(M,F)$}\label{sec:definition-of-F}

Let $X$ be an ordinary K3 surface over $\bF_q$.  The following theorem, due to Nygaard and Yu, says that the Frobenius on $X$ can be lifted to an endomorphism of the Betti cohomology of $X_\can^\iota$, at least after inverting $p$. The proof relies on the Kuga--Satake construction.

\begin{theorem}[Nygaard {\cite[\S~3]{Nygaard83}}, Yu {\cite[Lemma~2.3]{Yu12}}]
\label{theorem:definition-of-F}
There is a unique endomorphism $F$ of $\rH^2(X_\can^\iota,\,\bZ[\tfrac{1}{p}])$ such that
\begin{enumerate}
\item for every $\ell\neq p$ the map $F$ corresponds under the comparison isomorphism
\[
	\rH^2(X_\can^\iota,\,\bZ[\tfrac{1}{p}]) \otimes \bZ_\ell 
	\longisomto \rH^2_\et(X_{\bar \bF_q}, \,\bZ_\ell)
\]
to the geometric Frobenius $\Frob$ on \'etale cohomology,
\item the map $F$ corresponds under the comparison isomorphism
\[
	\rH^2(X_\can^\iota,\,\bZ[\tfrac{1}{p}]) \otimes B_\crys \longisomto 
	\rH^2_\crys(X/W) \otimes_W B_\crys
\]
to the endomorphism $\phi^a \otimes \id$, where $\phi$ denotes the crystalline Frobenius. 
\end{enumerate}
Moreover, $F$ preserves the Hodge structure on $\rH^2(X^\iota_\can,\,\bQ)$. \qed
\end{theorem}

For later use in the proofs of Theorems \ref{bigthm:properties} and \ref{bigthm:equivalence}, we record some well-known properties of Tate twists of unramified $p$-adic Galois representations.

\begin{lemma} \label{lemma:pure-crys}
Let $V$ be an unramified $\Gal_K$-representation over $\bQ_p$, and let $n$ be an integer. Let $\Frob$ be the geometric Frobenius endomorphism of $V$, relative to $\bF_q$. Consider the $\Gal_K$-module $V(-n) := V\otimes \bQ_p(-n)$. 
\begin{enumerate}
\item The map
\[
	K^\times \to \GL( V(-n) ),\, x \mapsto \Frob^{v(x)} \,\otimes\, q^{nv(x)}\Nm_{K/\bQ_p}(x)^{-n}
\]
factors over the reciprocity map $K^\times \to \Gal_K^\ab$ and induces the action of $\Gal_K$ on $V(-n)$.
\item $D_\crys(q^n \Frob)=\phi^a$ as endomorphisms
of $D_\crys(V(-n))$.
\end{enumerate}
\end{lemma}

\begin{proof}
The first statement follows from Lubin--Tate theory, see for example \cite[\S~3.1, Theorem~2]{SerreLCFT}. For the second, one uses that the functor $D_\crys$ commutes with Tate twists to reduce to the case $n=0$. In this case, the statement follows from the observation
that $\Frob\otimes 1 = 1 \otimes \phi^a$ as endomorphisms of
\[
	\big( V \otimes_{\bZ_p} W(\bar \bF_q) \big)^{\Gal_{\bF_q}},
\]
where the action of $\Gal_{\bF_q}$ is the diagonal one.
\end{proof}


Theorem \ref{bigthm:properties} is now an almost immediate consequence of
Theorem \ref{bigthm:etale-characterisation}.

\begin{proof}[Proof of Theorem \ref{bigthm:properties}]
By its definition (in Theorem \ref{theorem:definition-of-F}), the endomorphism $F$ of $\rH^2(X_\can^\iota,\,\bQ) \otimes \bQ_p = \rH^2_\et(X_{\bar K},\,\bQ_p)$ satisfies $D_\crys(F) = \phi^a$ on $\rH^2_\crys(X/W)[\tfrac{1}{p}]$. 

By Theorem \ref{bigthm:etale-characterisation} we have a decomposition
$\rH^2_\et(X_{\bar K},\bZ_p) = H^0 \oplus H^1(-1) \oplus H^2(-2)$ with $H^i$ unramified and by Lemma \ref{lemma:pure-crys}, also the
endomorphism
\[
	F' := \Frob_{H^0} \,\oplus\, q\Frob_{H^1} \,\oplus\, q^2 \Frob_{H^2}
\]
of $\rH^2_\et(X_{\bar K},\bQ_p)$ satisfies $D_\crys(F') = \phi^a$. Since $D_\crys$ is fully faithful, we must have $F=F'$. But it then follows
immediately that $F$ preserves the $\bZ_p$-lattice $\rH^2_\et(X_{\bar K},\,\bZ_p)$, and that $\rH^2_\et(X_{\bar K},\,\bZ_p)$ decomposes as described in the theorem.
\end{proof}

%

We thus have constructed from $X/\bF_q$ an integral lattice $M:=\rH^2(X_\can^\iota,\,\bZ)$, equipped with an endomorphism $F$, 
satisfying (M1)--(M4). We end this paragraph by relating the $p$-adic decomposition in (M4) to the Hodge decomposition for $X_\can^\iota$.

\begin{lemma}\label{lemma:hodge-structure}
Let $(M,F)$ be a pair satisfying (M1)--(M4). Then
complex conjugation on $M\otimes \bC$ maps the subspace $M^s \otimes_{\bZ_p,\iota} \bC$ to $M^{2-s} \otimes_{\bZ_p,\iota} \bC$.
\end{lemma}

In other words: the decomposition in (M4) induces under $\iota\colon \bZ_p\to \bC$ a $\bZ$-Hodge structure on $M$.

\begin{proof}[Proof of Lemma \ref{lemma:hodge-structure}]
By (M4), the one-dimensional subspace $M^0 \otimes_{\bZ_p,\iota} \bC$ of $M\otimes \bC$ is the unique eigenspace for the endomorphism $F$ corresponding to an eigenvalue $u \in \bQ_p \subset \bC$ with $v_p(u)=0$. By (M3) we have $u \bar u = q^2$, and hence 
also the eigenvalue $\bar u = q^2/u$ lies in $\bQ_p \subset \bC$. Since $v_p(\bar u)=v_p(q^2)$, we see that the corresponding eigenspace is $M^2 \otimes_{\bZ_p,\iota} \bC$. Similarly, complex conjugation maps $M^2$ to $M^0$. By (M2), the subspace $M^{1}$ is the orthogonal complement of $M^0\oplus M^2$, and hence is preserved by complex conjugation.
\end{proof}

\begin{proposition}\label{prop:F-Hodge}
Let $(M,F)$ be the pair associated to an ordinary K3 surface $X$ over $\bF_q$. Then
we have $M^{s} \otimes_{\bZ_p,\iota} \bC = \rH^{s,2-s}(X_\can^\iota)$ as subspaces of $M\otimes \bC = \rH^2(X_\can^\iota,\,\bC)$.
\end{proposition}

\begin{proof}Indeed, under the `Hodge--Tate' comparison isomorphism
\[
	 \rH^2_\dR(X_{\can,K}/K) \otimes_K \bC_p
	 \longisomto \rH^2_\et(X_{\bar K},\,\bQ_p) \otimes_{\bQ_p} \bC_p
	 \cong M \otimes_{\bZ_p} \bC_p 
\]
the subspace $(\Fil^i \rH^2_\dR(X_{\can,K}/K)) \otimes_K \bC_p$ is mapped to $\oplus_{s \geq i} M^s \otimes_{\bZ_p} \bC_p$. Extending $\iota$ to an embedding $\bC_p \to \bC$  we see that the Hodge filtration on $\rH^2(X_\can^\iota,\, \bC)$ agrees with the filtration on $M\otimes \bC$ induced by the decomposition on $M\otimes \bZ_p$, and hence by Lemma \ref{lemma:hodge-structure} we have $M^s \otimes_{\bZ_p} \bC = \rH^{s,2-s}(X_\can^\iota)$.
\end{proof}

\subsection{Line bundles and ample cone}\label{sec:line-bundles}

Let $X$ be an ordinary K3 surface over $\bF_q$. Recall from \S~\ref{sec:canonical-lifts} that line bundles on $X$ extend uniquely to $X_\can$.
We obtain isomorphisms
\[
	\Pic X \isomto \Pic X_{\can} \isomto \Pic X_{\can,K}. 
\] 
Let $K^\nr$ be the maximal unramified extension of $K$. 

\begin{proposition}\label{prop:ample-classes}
We have natural isomorphisms
\[
	\Pic X_{\bar\bF_q} \isomto \Pic X_{\can,K^\nr} \isomto \Pic X_{\can,\bar K},
\]
and a class $\lambda \in \Pic X_{\bar\bF_q}$ is ample if and only if its image in $\Pic X_{\can,\bar K}$ is ample.
\end{proposition}

\begin{proof}
The first isomorphism follows from  the fact that 
canonical lifts commute with finite unramified extensions. The second isomorphism follows from the triviality of the action of $\Gal(K/K^\nr)$ on $\Pic X_{\can,\bar K} \subset \rH^2_\et(X_{\can,\bar K},\bQ_\ell(1))$  and the vanishing of $\Br K^\nr$.

It remains to show that the isomorphism $\Pic X_{\bar \bF_q} \isomto \Pic X_{\can,\bar K}$ restricts to a bijection between the subsets of ample classes. Fix an ample line bundle $H$ on $X$. Then by the structure theorem on the ample cone of a K3 surface over an algebraically closed field (\cite[\S~8.1]{HuybrechtsK3}) we have that a line bundle $L$ on $X_{\bar \bF_q}$ is ample if and only if
\begin{enumerate}
\item $L^2>0$
\item for every $D \in \Pic X_{\bar \bF_q}$ with $D^2=-2$ we have $L\cdot D \neq 0$ and $L\cdot D$ has the same sign as $H\cdot D$
\end{enumerate}
and similarly for line bundles on $X_{\can,\bar K}$. 
But the bijection  $\Pic X_{\bar \bF_q} \isomto \Pic X_{\can,\bar K}$ is an isometry, and the canonical lift $H_\can$ of the ample line bundle $H$ is itself ample, so we conclude that the bijection preserves ample classes.
\end{proof}

\begin{proposition}\label{prop:F-Pic}
For every $d\geq 1$ the map
\[
	\Pic X_{\bF_{\!q^d}} \to \{ \lambda \in \rH^2(X_\can^\iota,\,\bZ) \mid F^d\lambda=q^d\lambda  \}
\]
is an isomorphism.
\end{proposition}

\begin{proof}
Injectivity is clear, it suffices to show that the map is surjective. Without loss of generality we may assume that $d=1$. 

By Proposition \ref{prop:F-Hodge} any $\lambda \in \rH^2(X_\can^\iota,\,\bZ)$  satisfying $F\lambda = q\lambda$ is a Hodge class and by Theorem \ref{theorem:definition-of-F} we see that $\lambda$ defines a $\Gal_K$-invariant element of $\Pic X_{\can,\bar K}$. By the previous proposition, $\lambda$ corresponds to a  $\Gal_{\bF_q}$-invariant class in $\Pic X_{\bar\bF_q}$, which defines a line bundle $\cL$ on $X$ since the Brauer group of $\bF_q$ vanishes. We conclude that the map is surjective as claimed.
\end{proof}

\begin{proposition}
The real cone $\cK \subset M\otimes_\bZ R$ spanned by the classes of ample line bundles on $\Pic X_{\bar\bF_q}$ satisfies (M5). 
\end{proposition}

\begin{proof}
This follows immediately from the Propositions \ref{prop:ample-classes} and \ref{prop:F-Pic}, and the structure of the ample cone of a complex K3 surface.
\end{proof}

\subsection{Fully faithfulness}\label{sec:fully-faithful}
In \S~\ref{sec:definition-of-F} and \S~\ref{sec:line-bundles} we have constructed a functor $X \mapsto (M,F,\cK)$ from ordinary K3 surfaces over $\bF_q$ to triples satisfying (M1)--(M5). We end this section by showing that this functor is fully faithful.

\begin{proof}[Proof of fully faithfulness in Theorem~\ref{bigthm:equivalence}]
This is shown in \cite{Nygaard83b} and \cite[Theorem~3.3]{Yu12} for K3 surfaces equipped with an ample line bundle. The same argument works here, we repeat it for the convenience of the reader.

\medskip\noindent\emph{Faithfulness}.
Assume that $f,g\colon X_1 \to X_2$ are morphisms between ordinary K3 surfaces inducing the same maps $\rH^2(X_{2,\can}^\iota,\,\bZ) \to \rH^2(X_{1,\can}^\iota,\,\bZ)$. Then $f_\can^\iota = g_\can^\iota$ as maps from $X_{1,\can}^\iota$ to $X_{2,\can}^\iota$ and therefore $f_\can=g_\can$ and $f=g$. 

\medskip\noindent\emph{Fullness}.
Let $X_1$ and $X_2$ be ordinary K3 surfaces over $\bF_q$. Let 
\[
	\varphi \colon \rH^2(X_{2,\can}^\iota,\,\bZ) \to \rH^2(X_{1,\can}^\iota,\,\bZ)
\]
be an isometry commuting with $F$ and respecting ample cones. By the description of the ample cones of $X_1$ and $X_2$, we may choose ample line bundles $\cL_1$ and $\cL_2$ on $X/\bF_q$ such that $\varphi$ maps $c_1(\cL_{2,\can}^\iota)$ to $c_1(\cL_{1,\can}^\iota)$.

By Proposition~\ref{prop:F-Hodge} the map $\varphi$ respects the Hodge structures, and therefore the Torelli theorem shows that there is a unique isomorphism $f\colon X_{1,\can}^\iota \isomto X_{2,\can}^\iota$ with $f^\ast = \varphi$. Since $f^\ast F_2 = F_1 f^\ast$, and since
the \'etale cohomology of the $X_{i,\can,\bar K}$ is unramified, we have that
\[
	f^\ast\colon \rH^2_\et(X_{2,\can,\bar K},\,\bQ_\ell) \to \rH^2_\et(X_{1,\can,\bar K},\,\bQ_\ell)
\]
is $\Gal_K$-equivariant, and hence $f$ descends to a morphism of polarized K3 surfaces over $K$.  By Matsusaka--Mumford \cite[Thm.~2]{MatsusakaMumford64} this extends to an isomorphism $f\colon X_{1,\can} \isomto X_{2,\can}$ and we conclude that $\varphi$ comes from an isomorphism $f\colon X_1 \isomto X_2$ over $\bF_q$.
\end{proof}

\section{Essential surjectivity}

\subsection{Models of K3 surfaces with complex multiplication}\label{sec:CM-models}
We briefly recall a few facts about complex K3 surfaces with complex multiplication. We refer to \cite{Zarhin83} for proofs. Let $X/\bC$ be a K3 surface. Its ($\bQ$-)transcendental lattice $V_X$ is defined as the orthogonal complement of $\NS(X) \otimes_\bZ \bQ$ in $\rH^2(X,\,\bQ(1))$. The endomorphism algebra $E$ of the $\bQ$-Hodge structure $V_X$  is a field, and we say that $X$ has \emph{complex multiplication} (by $E$) if $V_X$ is one-dimensional as an $E$-vector space. In this case, $E$ is necessarily a CM-field. Denote its complex conjugation by $\sigma\colon E\to E$. The Mumford-Tate group of the Hodge structure $V_X$ is the algebraic torus $T/\bQ$ defined by
\[
	T(A) = \{ x \in (A \otimes E)^\times \mid x\sigma(x)=1 \}
\]
for all $\bQ$-algebras $A$. Note that $T$ is an algebraic subgroup of $\SO(V_X)$.

If $X/\bC$ is a K3 surface with CM by $E$, then it can be defined over a number field. In \cite[Theorem~2]{Taelman17} we classified the models $\cX$ of $X$ over finite extensions $F$ of $E$ in terms of their Galois representations on $\rH^2_\et(\cX_{\bar F},\hat\bZ)$. We will deduce from that result a version for models over local fields. In the statement we will need the composition
\[
	\rec\colon \Gal_E^\ab \cong \bA_{E,f}^\times/E^\times \longto T(\bA_f)/T(\bQ),
\]
where the isomorphism is given by global class field theory (note that $E$ has no real places), and the second map is given by $z \mapsto \frac{z}{\sigma(z)}$.

\begin{theorem}\label{thm:local-descent}
Let $X/\bC$ be a K3 surface with CM by $E$. Let $K$ be a $p$-adic field containing $E$, and fix an embedding $\iota\colon K \to \bC$ extending the embedding $E\to \bC$ given by the action of
$E$ on $\rH^{2,0}(X)$. Let
\[
	\rho \colon \Gal_K \to \rO(\rH^2(X,\hat\bZ(1)))
\]
be a continuous homomorphism.  Assume that 
for every $\sigma \in \Gal_K$ we have
\begin{enumerate}
\item $\rho(\sigma)$ stabilizes $\NS(X) \subset  \rH^2(X,\hat\bZ(1))$ and $\cK_X \subset \NS(X)\otimes \bR$,
\item the restriction of $\rho(\sigma)$ to the transcendental lattice lands in the subgroup $T(\bA_f) \subset \rO(V_X \otimes \bA_f)$ and its image in $T(\bA_f)/T(\bQ)$ is  $\rec(\sigma)$.
\end{enumerate}
Then there exists a model $\cX/K$ of $X$ so that the resulting action of $\Gal_K$ on  $\rH^2_\et(\cX_{\bar K},\hat\bZ(1)) = \rH^2(X,\hat\bZ(1))$ coincides with $\rho$.
\end{theorem}

We can reformulate the conditions on $\rho$ as follows. Denote by $\Gamma \subset \rO(\rH^2(X,\hat\bZ(1)))$ the subgroup consisting of those $g$ satisfying
\begin{enumerate}
\item $g$  stabilizes $\NS(X) \subset  \rH^2(X,\hat\bZ(1))$ and $\cK_X \subset \NS(X)\otimes \bR$,
\item the induced action on $V_X \otimes \bA_f$ factors over $T(\bA_f)$,
\end{enumerate}
Then $\rho \colon \Gal_K \to \rO(\rH^2(X,\hat\bZ(1)))$ satisfies the conditions in the theorem if and only if it factors over 
$\Gamma$, and makes the square
\[
\begin{tikzcd}
\Gal_K \arrow{r} \arrow{d}{\rho} & \Gal_E \arrow{d}{\rec} \\
\Gamma \arrow{r}{\delta} & T(\bA_f)/T(\bQ)
\end{tikzcd}
\]
commute. In particular, we can consider $\rho$ as a lift of the map $\rec$. This point of view will be useful in the proof of Theorem~\ref{thm:local-descent}.

\begin{lemma}The map $\delta\colon \Gamma \to T(\bA_f)/T(\bQ)$ is a continuous open homomorphism of profinite groups with finite kernel. 
\end{lemma}

\begin{proof}
Let $\Gamma_0 \subset \Gamma$ be the open subgroup of finite index consisting of those elements that act trivially on $\NS(X)$. The map $\Gamma_0 \to T(\bA_f)$ is injective, and identifies $\Gamma_0$ with a compact open subgroup. Let $\cK$ be the maximal compact open subgroup. Then it suffices to show that $\cK \to T(\bA_f)/T(\bQ)$ has finite kernel and cokernel. The kernel is
 $\{ x \in \cO_E^\times  \mid x\sigma(x) = 1 \}$, which is finite because $E$ is a CM-field with complex conjugation $\sigma$. The finiteness of the cokernel  $T(\bQ) \backslash T(\bA_f)/\cK$ is a property of arbitrary tori over $\bQ$, see \cite[Prop.~9 \& Thm.~2]{Ono59}.
\end{proof}

\begin{lemma}\label{lemma:lifting}
Let $\delta\colon G_0 \to G_1$ be an open continuous homomorphism of profinite groups with finite kernel. Let $H$ be a closed subgroup of $G_1$, and $\rho\colon H\to G_0$ a continuous momorphism making the triangle
\[
\begin{tikzcd}
H \arrow[swap]{d}{\rho} \arrow[hook]{rd} &   \\
G_0 \arrow{r}{\delta} &  G_1
\end{tikzcd}
\]
commute.  Then there exists an open subgroup $U\subset G_1$ containing $H$, and a continuous homomorphism $\rho'\colon U \to G_0$ making the square
\[
\begin{tikzcd}
H \arrow[swap]{d}{\rho} \arrow[hook]{r} & U \arrow[dashed,swap]{ld}{\rho'} \arrow[hook]{d}   \\
G_0 \arrow{r}{\delta} &  G_1
\end{tikzcd}
\]
commute.
\end{lemma}

\begin{proof}
Since the kernel of $\delta$ is finite, there exists an open subgroup $H_0 \subset G_0$ with $\ker \delta \cap H_0=\{1\}$. The normalizer of $H_0$ has finite index in $G_0$, so the intersection $N_0 := \bigcap_{g\in G_0} gH_0 g^{-1}$ is a normal open subgroup on which $\delta$ is injective.  Denote by $N_1 \subset G_1$ its image, and by $s\colon N_1 \isomto N_0$ the inverse isomorphism. Note that $N_1\subset G_1$ is open, and normalized by $H\subset G$.

Consider the continuous function
\[
	N_1 \cap H \to  G_0,\, g \mapsto s(g) \rho(g)^{-1}.
\]
It takes values in the finite subset $\ker \delta \subset G_0$, and maps $1$ to $1$. The collection of 
subgroups of the form $N_1 \cap H$ (for varying normal open $N_0 \subset G_0$) is a basis for the topology on $H$, so shrinking $N_0$ if necessary, we may without loss of generality assume that the above map is constant. We then have $s(g) = \rho(g)$ for all $g\in N_1 \cap H$.

The product $U := N_1 \cdot H$ is an open subgroup of $G$ containing $H$, and by the above  the map
\[
	\rho'\colon U \to G_0,\, gh \mapsto s(g) \rho(h) \quad\quad\quad ( g\in N_1,\,\, h \in H )
\]
is a well-defined homomorphism satisfying the required properties.
\end{proof}

\begin{proof}[Proof of Theorem \ref{thm:local-descent}]
Let $G_0$ be the topological group
defined by the cartesian square
\[
\begin{tikzcd}
G_0 \arrow{r}{\delta'} \arrow{d} & \Gal_E \arrow{d}{\rec} \\
\Gamma \arrow{r}{\delta} & T(\bA_f)/T(\bQ)
\end{tikzcd}
\]
Note that $\delta'$ is open with finite kernel, and hence that $G_0$ is a profinite group. Now let $\rho$ be as in the statement of the theorem. Then it induces a commutative triangle
\[
\begin{tikzcd}
\Gal_K \arrow[swap]{d}{\rho'} \arrow[hook]{rd} &   \\
G_0 \arrow{r}{\delta'} &  \Gal_E.
\end{tikzcd}
\]
By Lemma \ref{lemma:lifting} there exists an intermediate field  $E\subset F \subset K$ with $F$ finite over $E$, and a continuous homomorphism $\rho'' \colon \Gal_F \to G_0$ making the diagram
\[
\begin{tikzcd}
\Gal_K \arrow[swap]{d}{\rho'} \arrow{r}  & \Gal_F \arrow[hook]{d} \arrow[swap]{ld}{\rho''}  \\
G_0 \arrow{r}{\delta'} &  \Gal_E.
\end{tikzcd}
\]
commute. Now \cite[Theorem~2]{Taelman17} guarantees the existence of a model $\cX$ over $F$
whose Galois action on $\rH^2_\et(\cX_{\bar F},\,\hat\bZ(1))=\rH^2(X,\,\hat\bZ(1))$ is given by the composition
\[
	\Gal_F \overset{\rho''}\longto G_0 \longto \Gamma,
\]
and hence the
base change of $\cX$ to $K$ fulfils the requirements. 
\end{proof}

\subsection{Criteria of good reduction}
Let $\cO_K$ be a discrete valuation ring with fraction field $K$ and perfect residue field $k$. In the introduction we defined a property ($\star$) for K3 surfaces over $K$.

\begin{theorem}[Liedtke--Matsumoto \cite{LiedtkeMatsumoto18}]
\label{thm:LM}
Let $X$ be a K3 surface over $K$ satisfying $(\star)$. If for some $\ell$ different from the characteristic of $k$ the action of $\Gal_K$ on $\rH^2_\et(X_{\bar K},\,\bZ_\ell) $
is unramified, then there exists a finite unramified extension $K \subset K'$ and a proper smooth
algebraic space $\fX'$ over $\cO_{K'}$ with $\fX_{K'} \cong X_{K'}$.\qed
\end{theorem}

Analyzing the proof in the case where the specialization map on Picard groups is bijective, one obtains a stronger conclusion.

\begin{proposition}\label{prop:CLL}
Let $X$ be a K3 surface over $K$, let $K\subset K'$ be an unramified extension, and let $\fX'$ over $\cO_{K'}$ be a proper smooth algebraic space with $\fX'_{K'} \cong X_{K'}$. If the reduction map
$\Pic \fX'_{\bar K} \to \Pic \fX'_{\bar k}$
is bijective, then there exists a smooth projective $\fX$ over $\cO_K$ with $\fX_K \cong X$. 
\end{proposition}

\begin{proof}
The map $\Pic \fX'_{\bar K} \isomto \Pic \fX'_{\bar k}$ identifies the $(-2)$-classes on generic and special fiber, and hence induces a bijection between the ample cones in $\Pic \fX'_{\bar K}$ and $\Pic \fX'_{\bar k}$ (see also the proof of Proposition~\ref{prop:ample-classes}). Now choose an ample line bundle $\cL$ on $X$. It induces an ample line bundle $\cL'$ on $X':=X_{K'}$, which extends to a relatively ample line bundle on $\fX'$. In particular, the canonical RDP model 
$P(X',\cL')$ over $\cO_{K'}$ of Liedtke and Matsumoto \cite[Thm.~1.3]{LiedtkeMatsumoto18} is non-singular. It follows from the construction of this model that $P(X',\cL') = P(X,\cL) \otimes_{\cO_K} \cO_{K'}$ (see the end of section 6 in \cite{CLL}), and hence $P(X,\cL)$ is a smooth projective model $\fX$ over $\cO_K$. 
\end{proof}

Alternatively, one can verify that under the hypothesis of Proposition \ref{prop:CLL} the group $\cW^\nr_{X,\cL}$ occurring in \cite[Thm.~1.4]{CLL} vanishes.

\subsection{Proof of Theorem \ref{bigthm:equivalence}}\label{sec:surjectivity}

In \S~\ref{sec:fully-faithful} we have established that the functor $X/\bF_q \mapsto (M,F,\cK)$ is fully faithful. To finish the proof of Theorem \ref{bigthm:equivalence}, it remains to show that the functor is essentially surjective, assuming ($\star$) holds for K3 surfaces over $p$-adic fields.

\begin{proof}[Proof of essential surjectivity in Theorem \ref{bigthm:equivalence}]
Let $(M,F,\cK)$ be a triple satisfying (M1)--(M5). We will show that it lies in the essential image of our functor by constructing a suitable K3 surface over $\bF_q$. We divide the construction in several steps.

\medskip\noindent
\emph{Construction of a complex K3 surface}.
By Lemma~\ref{lemma:hodge-structure} the decomposition
\[
	M_\bC  = M^{2,0} \oplus M^{1,1} \oplus M^{0,2}
\]
with $M^{s,2-s} := M^s \otimes_{\bZ_p,\iota} \bC$ defines a $\bZ$-Hodge structure on $M$. By the Torelli theorem for complex K3 surfaces, there is a projective K3 
surface $X$ and a Hodge isometry $f\colon \rH^2(X,\,\bZ) \isomto M$ mapping the ample cone of $X$ to $\cK$. The pair $(X,f)$ is unique up to unique isomorphism.

\medskip\noindent
\emph{$X$ has complex multiplication}. Let $V_X \subset \rH^2(X,\,\bQ(1))$ be the transcendental lattice. Note that $F$ respects the decomposition $\rH^2(X,\,\bQ(1)) = \NS(X)\otimes \bQ \oplus V_X$.  Every $\bQ$-linear endomorphism of $V_X$ that commutes with $F$ will respect the Hodge structure on $V_X$, and since the endomorphism algebra of the $\bQ$-Hodge structure $V_X$ is a field, we conclude that $V_X$ is a cyclic $\bQ[F]$-module, that $E:= \End V_X$ is generated by $F$, and that $\dim_E V_X=1$. In particular, $X$ has complex multiplication by $E$, the field $E$ is then a CM field, and if we denote the complex conjugation on $E$ by $\sigma$, then the Mumford-Tate group $T$ of $V_X$ satisfies $T(\bQ) = \{ x\in E^\times \mid x\sigma(x)=1\}$.
Observe that $\sigma(F) = q^2/F$  on $V_X$, and hence that $F/q$ defines an element of $T(\bQ)$.

The number field $E$ has a unique place $v \mid p$  satisfying $v(F)>0$. We have $E_v=\bQ_p$. Let $K$ be the fraction field of $W(\bF_q)$, considered as a subfield of $\bC$ via $\iota$.

%
%

\medskip\noindent
\emph{Descent to $K\subset \bC$}. For every $\ell \neq p$ consider the unramified $\Gal_K$-representation
\[
	\rho_\ell\colon \Gal_K \to \GL( M\otimes \bZ_\ell) 
\]
given by letting the geometric Frobenius $\Frob$ act as $F$. We also define a $p$-adic $\Gal_K$-representation
\[
	\rho_p\colon \Gal_K \to \GL( M\otimes \bZ_p) = \GL( \oplus_s M^s )
\]
by declaring that the Tate twisted $\bZ_p[\Gal_K]$-modules $M^s(s)$ are unramified with geometric Frobenius $\Frob$ acting as $F/q^s$. 
The $\rho_\ell$ and $\rho_p$ assemble into an action of $\Gal_K$ on $M\otimes \hat\bZ$. Denote by 
$M\otimes \hat\bZ(1)$ its Tate twist. The resulting map
\[
	\rho \colon \Gal_K \to \GL(M\otimes \hat \bZ(1)) = \GL(\rH^2(X,\,\hat\bZ(1))
\]
satisfies
\begin{enumerate}
\item the image of $\rho$ is contained in $O(\rH^2(X,\,\hat\bZ(1)))$,
\item the image of $\rho$ preserves $\Pic X$ and the ample cone
$\cK \subset (\Pic X)\otimes \bR$.
\end{enumerate}
We claim that $\rho$ also satisfies the reciprocity condition in Theorem~\ref{thm:local-descent}. 

Indeed,
observe that the action of $\Gal_K$ on $M\otimes \hat\bZ(1)$ is abelian. Let $x\in K^\times$. Using Lemma~\ref{lemma:pure-crys} we see that the action of the corresponding $\tau = \tau(x) \in \Gal_K^\ab$ on $M\otimes \hat \bZ(1)$ satisfies
\begin{enumerate}
\item $\tau$ acts on $M\otimes \bZ_\ell(1)$ by $(F/q)^{v(x)}$ (for $\ell \neq p$),
\item $\tau$ acts on $M^0(1) \subset M\otimes\bZ_p(1)$ by $(\Nm_{K/\bQ_p} x) (F/q)^{v(x)}$,
\item $\tau$ acts on $M^1(1)\subset M\otimes\bZ_p(1)$ by $(F/q)^{v(x)}$,
\item $\tau$ acts on $M^2(1)\subset M\otimes\bZ_p(1)$ by $(\Nm_{K/\bQ_p} x)^{-1} (F/q)^{v(x)}$.
\end{enumerate}
(Note that by property (M4) these actions indeed preserve the $\bZ_p$-lattices $M^s(1)$).

On the other hand, the decomposition of $M\otimes \bZ_p$ induces a decomposition 
\[
	V_X \otimes \bQ_p = V_{-1} \oplus V_0 \oplus V_1
\]
with $\dim V_{-1} = \dim V_1 = 1$. The group $E\otimes\bQ_p$ acts on $V_X \otimes \bQ_p$ on $V_{-1}$ through the factor $E_v^\times\cong \bQ_p^\times$, and on $V_1$ trough $E_{\sigma v}^\times\cong \bQ_p^\times$. The inclusion of $E_v^\times \times E_{\sigma v}^\times \subset (E\otimes \bQ_p)^\times$ defines a subgroup
\[
	T_{v,\sigma v} = \{ (x_v,x_{\sigma v}) \in E_v^\times \times E_{\sigma v}^\times \, \mid \,
		x_v x_{\sigma v} = 1 \} \subset T(\bQ_p).
\]
The compatibility between  local and global class field theory and the definition of $\rec$ (see \S~\ref{sec:CM-models}) implies that the diagram
\[
\begin{tikzcd}
\Gal_K^\ab \arrow{rr}  & & \Gal_E^\ab \arrow{d}{\rec} \\
K^\times \arrow{u} \arrow{r} & T_{v,\sigma v} \arrow{r} & T(\bA_f)/T(\bQ)
\end{tikzcd}
\]
in which the map $K^\times \to T_{v,\sigma v}$ maps $x$ to $(\Nm_{K/\bQ_p}(x),\Nm_{K/\bQ_p}(x)^{-1})$
commutes. We conclude that
\[
	K^\times \longto \Gal_K^\ab \longto \Gal_E^\ab \overset{\rec}{\longto} T(\bA_f)/T(\bQ)
\]
maps an $x\in K^\times$ to the class of the element $\alpha = \alpha(x) \in T(\bA_f)$ satisfying
\begin{enumerate}
\item $\alpha_\ell=1$ for all $\ell \neq p$,
\item $\alpha_p$ acts on $V_X\otimes \bQ_p=V_{-1} \oplus V_{0} \oplus V_1$ by $(\Nm_{K/\bQ_p} x, 1, (\Nm_{K/\bQ_p} x)^{-1})$.
\end{enumerate}
Since $F/q$ lies in $T(\bQ)$, we see that $\alpha(x)$ and $\tau(x)$ define the same element in $T(\bA_f)/T(\bQ)$. This shows that $\rho$ satisfies the requirements of Theorem~\ref{thm:local-descent}, and we conclude that there is a model 
$\cX/K$ of $X$ whose $\Gal_K$-action on $\rH^2_\et(\cX_{\bar K},\,\hat\bZ) = M \otimes \hat\bZ$ is the prescribed one.

\medskip\noindent
\emph{Extension to $\cO_K$ and reduction to $k$.} By construction, the action of $\Gal_K$ on $\rH^2_\et(X_{\bar K},\,\bZ_\ell)$ is unramified. By Theorem \ref{thm:LM}, and since we are assuming $\cX$ satisfies ($\star$) there exists a finite unramified extension $K\subset K'$ so that $\cX' := \cX_{K'}$ has good reduction, and hence extends to a proper smooth $\fX'$ over $\cO_{K'}$. 

By Theorem~\ref{bigthm:etale-characterisation}, this model $\fX'$ is the canonical lift of its reduction, and hence by Proposition \ref{prop:line-bundles-lift} the map $\Pic \fX'_{\bar K} \to \Pic \fX_{\bar \bF_q}$ is surjective. We conclude with Proposition \ref{prop:CLL} that $X/K$ has a canonical smooth projective model $\fX/\cO_K$. Again Theorem \ref{bigthm:etale-characterisation} guarantees that $\fX$ is the canonical lift of its reduction $\fX_k$, and we see that the functor of Theorem \ref{bigthm:equivalence} maps $\fX_k$ to the given triple $(M,F,\cK)$.
\end{proof}

\subsection{Unconditional results}

As above, we fix an embedding $\iota\colon W(\bF_q) \to \bC$.

\begin{theorem}\label{thm:unconditional}
The functor $X \mapsto (M,F,\cK)$ restricts to an equivalence between the sub-groupoids consisting of: 
\begin{enumerate}
	\item  $X/\bF_q$ for which there is an ample $\cL \in \Pic X_{\bar \bF_q}$  with $\cL^2<p-4$,
	\item $(M,F,\cK)$ for which there exists a $\lambda \in M\cap \cK$
	satisfying $\lambda^2<p-4$.
\end{enumerate}
Assuming $p\geq 5$ it also restricts to an equivalence between 
\begin{enumerate}
	\item $X/\bF_q$ for which  $\Pic X_{\bar \bF_q}$ contains a hyperbolic lattice,
	\item $(M,F,\cK)$ for which $\NS(M,F)$ contains a hyperbolic lattice,
\end{enumerate}
and between
 \begin{enumerate}
	\item $X/\bF_q$ with $\rk \Pic X_{\bar \bF_q} \geq 12$,
	\item $(M,F,\cK)$  with $\rk \NS(M,F) \geq 12$.
\end{enumerate}
\end{theorem}

\begin{proof}
In view of Theorem~\ref{bigthm:equivalence} and Proposition~\ref{prop:ample-classes}, we only need to verify that any triple $(M,F,\cK)$ as in (ii) lies in the essential image of the functor $X\mapsto (M,F,\cK)$ on ordinary K3 surfaces. It suffices to show that the relevant $X$ over $K=\Frac W(\bF_q)$ occurring in the proof in \S~\ref{sec:surjectivity} satisfy ($\star$).

By \cite[Thm.~1.1]{Matsumoto15} and \cite[\S~2]{Ito17} we know that
any K3 surface over $X$ with unramified $\rH^2_\et(X_{\bar K},\,\bQ_\ell)$, and satisfying one of 
\begin{enumerate}
\item there is an ample $\cL \in \Pic X_{\bar K}$ with $\cL^2<p-4$,
\item $\Pic X_{\bar K}$ contains a hyperbolic plane and $p\geq 5$,
\item $\Pic X_{\bar K}$ has rank $\geq 12$ and $p\geq 5$,
\end{enumerate}
has potentially good reduction. In particular, any such K3 surface satisfies hypothesis ($\star$).  In all three cases the argument of \S~\ref{sec:surjectivity} goes through unconditionally.
\end{proof}

\end{document}